\documentclass[a4paper,11pt]{amsart}

\usepackage{amsmath, amsthm, amsfonts, amssymb}
\usepackage[bookmarks=false]{hyperref}
\usepackage{enumerate}

\usepackage{color}

\numberwithin{equation}{section}

\newtheorem{letterthm}{Theorem}

\newtheorem{lettercor}[letterthm]{Corollary}

\newtheorem{letterconj}[letterthm]{Conjecture}

\newtheorem{thm}{Theorem}[section]
\newtheorem{lem}[thm]{Lemma}

\newtheorem{prop}[thm]{Proposition}

\theoremstyle{definition}
\newtheorem{rem}[thm]{Remark}

\newtheorem{df}[thm]{Definition}

\newtheorem*{question}{Question}

\newcommand{\R}{\mathbf{R}}
\newcommand{\C}{\mathbf{C}}
\newcommand{\Z}{\mathbf{Z}}
\newcommand{\F}{\mathbf{F}}

\newcommand{\N}{\mathbf{N}}
\newcommand{\B}{\mathbf{B}}
\newcommand{\K}{\mathbf{K}}

\newcommand{\Ad}{\operatorname{Ad}}
\newcommand{\id}{\text{\rm id}}

\newcommand{\Aut}{\operatorname{Aut}}

\newcommand{\rL}{\mathord{\text{\rm L}}}

\newcommand{\dom}{\mathord{\text{\rm dom}}}

\newcommand{\conv}{\overline{\mathord{\text{\rm conv}}} \,}

\newcommand{\Ball}{\mathord{\text{\rm Ball}}}

\newcommand{\spec}{\mathord{\text{\rm spec}}}

\newcommand{\ovt}{\mathbin{\overline{\otimes}}}
\newcommand{\bdot}{\, \cdot \,}

\newcommand{\op}{^{\mathrm{op}}}
\newcommand{\rd}{\: \mathrm{d}}
\newcommand{\ri}{\mathrm{i}}

\newcommand{\II}{{\rm II}}
\newcommand{\III}{{\rm III}}

\begin{document}

\title[Full factors, bicentralizer flow and approximately inner automorphisms]{Full factors, bicentralizer flow\\ and approximately inner automorphisms}

\begin{abstract}
We show that a factor $M$ is full if and only if the $C^*$-algebra generated by its left and right regular representations contains the compact operators. We prove that the bicentralizer flow of a type $\III_1$ factor is always ergodic. As a consequence, for any type $\III_1$ factor $M$ and any $\lambda \in ]0,1]$, there exists an irreducible AFD type $\III_\lambda$ subfactor with expectation in $M$. Moreover, any type $\III_1$ factor $M$ which satisfies $M \cong M \ovt R_\lambda$ for some $\lambda \in ]0,1[$ has trivial bicentralizer. Finally, we give a counter-example to the characterization of approximately inner automorphisms conjectured by Connes and we prove a weaker version of this conjecture. In particular, we obtain a new proof of Kawahigashi-Sutherland-Takesaki's result that every automorphism of the AFD type $\III_1$ factor is approximately inner.
\end{abstract}

\address{Research Institute for Mathematical Sciences, Kyoto University, Kyoto, 606-8502 Japan}

\author{Amine Marrakchi}
\email{amine.marrakchi@math.u-psud.fr}

\thanks{The author is supported by JSPS}



\maketitle


\section{Introduction and statement of the main results}

\subsection*{Introduction} 
In this paper, we present several results on type $\III$ factors which are similar in the sense that they are all proved by using the same method. Indeed, the main technical novelty of this paper is a very general observation: given a self-adjoint operator $X$ on a Hilbert space $H$, we show that a state $\Psi \in \B(H)^*$ lies in the closed convex hull of all approximate eigenstates of $X$ if and only if it satisfies a strong invariance property with respect to the one-parameter group $(e^{\ri t X})_{t \in \R}$. This criterion provides a new method to deal with some issues which are specific to type $\III_1$ factors, when the modular operator has no true eigenvectors, and we hope that it will have further applications in the future.

\subsection*{Full factors} Following \cite{Co74}, we say that a factor $M$ is \emph{full} if it satisfies the following \emph{centralizing} net criterion: for every bounded net $(x_i)_{i \in I}$ in $M$ such that $\lim_i \| \varphi(x_i \bdot)-\varphi(\bdot x_i) \|=0$ for all $\varphi \in M_*$, there exists a bounded net $z_i \in \C$ such that $x_i - z_i \to 0$ in the strong topology. When $M$ is of type $\II_1$, this is equivalent to $M$ not having \emph{property Gamma} of Murray and von Neumann \cite{MvN43}. For example, Murray and von Neumann showed that the unique hyperfinite type $\II_1$ factor $R$ is not full while the free group factors $L(\F_n), \: n\geq 2$ are full. They hence obtained the first example of two (separable) non-isomorphic type $\II_1$ factors. 

In his famous paper on the classification of injective factors \cite{Co75b}, Connes established a powerful characterization of fullness. He showed that for any type $\II_1$ factor, the following three properties are equivalent:
\begin{itemize}
\item [$(\rm i)$] $M$ is full.
\item [$(\rm ii)$] The adjoint representation $\Ad : \mathcal{U}(M) \curvearrowright \rL^2(M)$ has spectral gap.
\item [$(\rm iii)$] The $C^*$-algebra $C^*_{\lambda \cdot \rho}(M)=\overline{\lambda(M)\rho(M)}^{\| \cdot \|}$ generated by the left and right regular representations $\lambda, \rho : M \rightarrow \B(\rL^2(M))$ contains the compact operators.
\end{itemize}
When $M$ is \emph{injective}, then $C^*_{\lambda \cdot \rho}(M)$ is isomorphic to $\lambda(M) \otimes_{\min} \rho(M)$ which is a simple $C^*$-algebra so that it cannot contain the compact operators. Hence, property $(\rm iii)$ was used by Connes to show that an injective $\II_1$ factor cannot be full and this was a key step in his proof of injectivity implies hyperfiniteness. He also used it to show that a tensor product of two $\II_1$ factors $M \ovt N$ is full if and only if $M$ and $N$ are both full. Finally, Connes' criterion is also an important tool in many deformation/rigidity arguments based on fullness, see for instance \cite[Theorem 5.1]{Po06} and \cite{Po10}. 

Unfortunately, the work of Connes does not generalize in a straightforward way to the type $\III$ situation. Indeed, it is easy to see that for \emph{any} infinite factor $M$ (whether it is full or not), the representation $\Ad : \mathcal{U}(M) \curvearrowright \rL^2(M)$ does not have invariant vectors and in fact, it does not have almost invariant vectors at all. Recently however, a form of spectral gap for full type $\III$ factors was obtained in \cite{Ma16}. While this characterization was sufficient for most applications \cite{HMV16}, \cite{Ma18}, it was not strong enough to answer the following question: is it true that for any full factor $M$, the $C^*$-algebra $C^*_{\lambda \cdot \rho}(M)$ contains the compact operators? In fact, only the type $\III_1$ case was left open and our first main theorem solves that question.

\begin{letterthm} \label{main full compact}
Let $M$ be an arbitrary factor. Then $M$ is full if and only if $C^*_{\lambda \cdot \rho}(M)$ contains the compact operators.
\end{letterthm}

A consequence of Theorem \ref{main full compact} is that one can characterize fullness in terms of \emph{central} nets.

\begin{lettercor} \label{cor central net}
A factor $M$ is full if and only if for every bounded net $(x_i)_{i \in I}$ in $M$ such that $x_ia-ax_i \to 0$ strongly for all $a \in M$, there exists a bounded net $(z_i)_{i \in I}$ in $\C$ such that $x_i-z_i \to 0$ strongly.
\end{lettercor}

There is another consequence of Theorem \ref{main full compact} regarding the topology of the automorphism group $\Aut(M)$. This group is usually equipped with the topology of pointwise norm convergence on the predual $M_*$, which means that a net $(\theta_i)_{i \in I}$ in $\Aut(M)$ converges to the identity if and only if $\lim_i \|\varphi - \varphi \circ \theta_i\|=0$ for all $\varphi \in M_*$. Another possible topology on $\Aut(M)$ is the topology of pointwise strong convergence on $M$, i.e.\ a net $(\theta_i)_{i \in I}$ converges to the identity if and only if $\theta_i(x)$ converges to $x$ strongly for all $x \in M$. These two topologies, called $u$-topology and $p$-topology respectively in \cite{Ha73}, are the same when $M$ is finite but they do not coincide in general when $M$ is of type $\III$ \cite[Corollary 3.15]{Ha73}. However, we have the following corollary of Theorem \ref{main full compact}.

\begin{lettercor} \label{cor topology}
Let $M$ be a full factor. Then the $u$-topology and the $p$-topology coincide on $\Aut(M)$.
\end{lettercor}

\subsection*{Bicentralizer flow}
This section is motivated by the following fundamental problem in the theory of von Neumann algebras.
\begin{question}
Let $M$ be a factor with separable predual. Can we find large amenable subfactors in $M$? More precisely, is there an amenable subfactor $P$ with a normal conditional expectation in $M$ such that $P' \cap M=\C$?
\end{question}
When $M$ is not of type $\III_1$, this problem was completely solved by Popa. More precisely, it follows from \cite{Po81}, that for any factor $M$ of type $\II_1$ or of type $\III_\lambda$ with $\lambda \in ]0,1[$, one can construct such a subfactor $P$ of type $\II_1$. If $M$ is of type $\II_\infty$, one can take $P$ of type $\II_\infty$ (but not of type $\II_1$). Finally, by \cite{Po83}, if $M$ is of type $\III_0$, one can take $P$ of type $\III_0$ (but not of type $\II$). The existence of such irreducible amenable subfactors is a very important tool, as one can often use it to reduce problems on general factors to the amenable case. See for instance \cite{Po18} for a very recent application of this idea.

When $M$ is of type $\III_1$, very little is known except that this problem is intimately related to Connes' bicentralizer problem which we now recall. Let $\varphi$ be a faithful normal state on $M$. The \emph{bicentralizer} of $\varphi$, denoted by $\B(M,\varphi)$ is the set of all elements $x \in M$ such that $\lim_n \|a_n x-x a_n\|_\varphi= 0$ for any bounded sequence $(a_n)_{n \in \N}$ in $M$ which satisfies $\lim_n \|\varphi(a_n \bdot)-\varphi(\bdot a_n)\|=0$. The bicentralizer $\B(M,\varphi)$ is a von Neumann subalgebra of $M$ and Connes' bicentralizer problem asks whether for any type $\III_1$ factor $M$, we have $\B(M,\varphi)=\C$. This is one of the most famous open questions in the theory of type $\III$ factors. By solving it when $M$ is amenable \cite{Ha85}, Haagerup settled the question of the uniqueness of the AFD type $\III_1$ factor which, after Connes' work, was the last remaining case in the classification program. In the same paper, inspired by Popa's technique, Haagerup also showed that a type $\III_1$ factor $M$ has trivial bicentralizer if and only if one can find an irreducible hyperfinite $\II_1$ factor with expectation in $M$, or if and only if there exists a faithful normal state $\varphi$ on $M$ which has a \emph{large centralizer}, i.e.\  $M_\varphi' \cap M=\C$.

When working on the bicentralizer problem, Connes realized (by using the Connes-St\o rmer transitivity theorem) that the bicentralizer $\B(M,\varphi)$ does not depend on the choice of the faithful normal state $\varphi$, up to canonical isomorphism. Therefore, one can speak of the \emph{canonical} bicentralizer $\B(M)$ of a given type $\III_1$ factor. This observation of Connes was recently enhanced in \cite[Theorem A]{AHHM18} in order to show the existence of a canonical flow $\beta: \R^*_+ \curvearrowright \B(M)$ called the \emph{bicentralizer flow}. Moreover, by adapting Popa and Haagerup ideas, it was shown in \cite[Theorem C]{AHHM18} that the bicentralizer flow $\beta : \R^*_+ \curvearrowright \B(M)$ captures all the information about the possible existence of irreducible AFD subfactors with expectation in $M$ and their possible types. In the second main theorem of this paper, we show that the bicentralizer flow is always ergodic. As a consequence, one can always find irreducible AFD type $\III$ subfactors with expectation in $M$.

\begin{letterthm} \label{main bicentralizer}
Let $M$ be a type $\III_1$ factor. Then $\beta_\lambda \curvearrowright \B(M)$ is ergodic for every $\lambda \in \R^*_+ \setminus \{1\}$. In particular, we have:
\begin{enumerate}
\item [$(\rm i)$] If $M$ has separable predual, then for every $\lambda \in ]0,1]$, we can find an AFD type $\III_\lambda$ subfactor with expectation $P \subset M$ such that $P ' \cap M =\C$.
\item [$(\rm ii)$] If $M \cong M \ovt R_\lambda$ for some $\lambda \in ]0,1[$, then $M$ has trivial bicentralizer.
\end{enumerate}
\end{letterthm}

We already explained the motivation for item $(\rm i)$ but we would like to also emphasize the importance of item $(\rm ii)$. Indeed, in the context of deformation/rigidity, it is often essential to have large centralizers in order to be able to use Popa's intertwining theory effectively. But, we observe (\cite[Lemma 6.1]{HMV16}) that for any inclusions with expectation $A,B \subset M$, one has $A \prec_M B$ if and only if $A \ovt R_\infty \prec_{M \ovt R_\infty} B \ovt R_\infty$ where $\prec_M$ is Popa's intertwining notation, and by item $(\rm ii)$, all these algebras have large centralizers. One can thus carry out the deformation/rigidity argument in the stabilized algebra $M \ovt R_\infty$ and then deduce the result in $M$. For example, one can in this way remove the large centralizer restrictions in \cite[Definition 1.1 and Theorem A]{Is17}. See also \cite[Application 4]{AHHM18}.

Item $(\rm ii)$ should also be compared with \cite[Theorem 3.5]{HI15} which shows that if the answer to the bicentralizer problem is negative then there exists a counter-example $M$ which is \emph{McDuff}, i.e.\ $M \cong M \ovt R$ where $R$ is a hyperfinite $\II_1$ factor. 

Finally, let us point out that a possible strategy to solve the bicentralizer problem is to show that $\beta_\lambda$ is approximately inner for some $\lambda \in ]0,1[$. Indeed, by \cite[Theorem B]{AHHM18}, this would automatically imply that $\beta_\lambda$ is trivial, hence that the bicentralizer itself is trivial since $\beta_\lambda$ is ergodic. This strategy was our original motivation for the next section.

\subsection*{Approximately inner automorphisms}
Let $M$ be a factor. We say that an automorphism $\theta \in \Aut(M)$ is \emph{weakly inner}\footnote{The terminology is justified by the fact that this property is equivalent to the correspondance $\rL^2(\theta)$ being \emph{weakly contained} in $\rL^2(M)$.} if there exists an automorphism $\alpha$ of $C^*_{\lambda \cdot \rho}(M)$ such that $\alpha(\lambda(a)\rho(b))=\lambda(\theta(a))\rho(b)$ for all $a,b \in M$.

It is easy to see that any approximately inner automorphism is weakly inner. On the other hand, when $M$ is full, it follows from Theorem \ref{main full compact}, that any weakly inner automorphism must be inner (see Proposition \ref{prop weak inner full}). When $M$ is an injective factor, every automorphism is weakly inner because in that case, one can take $\alpha=\theta \otimes \id$ on $C^*_{\lambda \cdot \rho}(M) \cong \lambda(M) \otimes_{\min} \rho(M)$.  In \cite{Co75b}, Connes proved the fundamental result that when $M$ is a $\II_1$ factor, an automorphism of $M$ is weakly inner if and only if it is approximately inner. In particular, every automorphism of an injective $\II_1$ factor is approximately inner. This was another key step to prove that injectivity implies hyperfiniteness. Connes also used this property to show that for any $\II_1$ factors $M$ and $N$, we have that $\theta_1 \otimes \theta_2 \in \Aut(M \ovt N)$ is approximately inner if and only if both $\theta_1$ and $\theta_2$ are approximately inner.

In the infinite case however, a weakly inner automorphism need not be approximately inner. For example, if $M$ is a $\II_\infty$ factor, then an automorphism which scales the trace cannot be approximately inner (but it is weakly inner if $M$ is injective). More generally, every automorphism $\theta$ of a factor $M$ induces an automorphism $\mathrm{mod}(\theta)$ of the \emph{flow of weights} of $M$ \cite{CT76} and it is easy to see that if $\theta$ is approximately inner then $\mathrm{mod}(\theta)$ must be trivial. Hence, motivated by the problem of the uniqueness of the AFD type $\III_1$ factor, Connes was very naturally led to the following conjecture.

\begin{letterconj}[{\cite[Section IV]{Co85}}] \label{connes conj}
Let $M$ be a factor and $\theta$ an automorphism of $M$. Then $\theta$ is approximately inner if and only if it is weakly inner and $\mathrm{mod}(\theta)$ is trivial.
\end{letterconj}
Note that in the case of type $\III_1$ factors, the flow of weights is trivial and therefore $\mathrm{mod}(\theta)$ is always trivial. So Conjecture \ref{connes conj} would imply in particular that any weakly inner automorphism of a type $\III_1$ factor is approximately inner. In particular, it would solve the bicentralizer problem because we know that the bicentralizer flow is weakly inner (in fact it satisfies a stronger property \cite[Proposition 4.1.$(\rm iv)$]{AHHM18} which seems to be closely related to the notion of ``locally approximately inner" automorphism of a $\II_\infty$ factor introduced in \cite[Definition 2.4]{Po93}).

Conjecture \ref{connes conj} has been verified for injective factors in \cite{KST92}. In \cite[Section IV]{Co85}, Connes claimed that he could verify it for type $\III_1$ factors with trivial bicentralizer and in \cite{GM93}, the authors claimed, also without proof, that they could verify it for factors of type $\III_\lambda, \: \lambda \in ]0,1[$. Unfortunately, it turns out that Conjecture \ref{connes conj} is not true. Indeed, let $N$ be a full factor and $P$ an amenable factor and let $M=N \ovt P$. Then for any $\theta \in \Aut(P)$, the automorphism $\id \otimes \theta \in \Aut(M)$ is weakly inner. Since $N$ is full, we know by \cite[Theorem A]{HMV16} that $\id \otimes \theta$ is approximately inner if and only if $\theta$ itself is approximately inner. Therefore, if $N$ is of type $\III_1$ and $\mathrm{mod}(\theta)$ is non-trivial, then $M$ is a type $\III_1$ factor and $\id \otimes \theta$ is weakly inner but not approximately inner. One can also construct in this way counter-examples of type $\III_\lambda$ for $\lambda \in ]0,1[$. We de not know if there exists type $\III_0$ counter-examples.

This shows that there might not be any simple criterion to determine when a weakly inner automorphism is approximately inner. Instead, we show in the next main theorem that these two properties become equivalent once we stabilize by the AFD type $\III_1$ factor $R_\infty$.

\begin{letterthm} \label{main weakly inner}
Let $M$ be a type $\III_1$ factor such that $M \cong M \ovt R_\infty$. Then every weakly inner automorphism of $M$ is approximately inner.

In particular, for any factor $N$ we have
$$ \theta \in \Aut(N) \text{ is weakly inner} \; \Leftrightarrow \; \theta \otimes \id \in \Aut(N \ovt R_\infty) \text{ is approximately inner}.$$
\end{letterthm}
This result gives nothing new for the bicentralizer flow, as one can already show that $\beta_\lambda \otimes \id \curvearrowright \B(M) \ovt R_\lambda$ is approximately inner for every $\lambda \in ]0,1[$ (see Proposition \ref{prop bic approx}). Therefore, for now, we do not have any application of Theorem \ref{main weakly inner} besides the fact that it gives a new and more direct proof of the following result of Kawahigashi, Sutherland and Takesaki.

\begin{lettercor}[KST92] \label{cor injective} Every automorphism of the AFD type $\III_1$ factor $R_\infty$ is approximately inner.
\end{lettercor} 

We also mention that for $\lambda \in ]0,1[$, we have a type $\III_\lambda$ analogue of Theorem \ref{main weakly inner} (see Theorem \ref{thm weak inner lambda}) which also implies that an automorphism $\theta$ of $R_\lambda$ is approximately inner if and only if $\mathrm{mod}(\theta)$ is trivial. The type $\III_0$ case remains unclear.

\subsection*{Acknowledgments} We are very grateful to Hiroshi Ando and Cyril Houdayer for all their valuable comments which greatly improved the exposition of this paper. We also thank Sorin Popa for pointing out the relevance of \cite[Definition 2.4]{Po93}. Finally, we thank Aymeric Baradat for his useful remark on $\rL^1(\R)$.

\tableofcontents

\section{Notations and preliminaries} 
\subsection*{Basic notations}
Let $M$ be any von Neumann algebra. We denote by $M_\ast$ its predual, by $\mathcal U(M)$ its group of unitaries and by $\mathcal Z(M)$ its center. The unit ball of $M$ is denoted by $\Ball(M)$. If $\varphi \in M_*^+$ is a positive functional, we put $\| x \|_\varphi=\varphi(x^*x)^{1/2}$ for all $x \in M$. In this paper, all von Neumann algebras are assumed to be $\sigma$-finite (they carry a faithful normal state). However, we never assume separability of the predual unless explicitly stated.

\subsection*{Standard form}  Let $M$ be any von Neumann algebra. We denote by $\rL^2(M)$ the standard form of $M$ \cite{Ha73} and by $\lambda, \rho : M \rightarrow \B(\rL^2(M))$ the left and right regular representation of $M$ on $\rL^2(M)$. We have $\lambda(M)'=J\lambda(M)J=\rho(M)$ where $J : \rL^2(M) \rightarrow \rL^2(M)$ is the canonical antilinear involution. We let $C^*_{\lambda \cdot \rho}(M)$ denote the $C^*$-algebra generated by the $*$-algebra $\lambda(M)\rho(M)$. Note that $\rho$ is an anti-representation of $M$ (or a representation of $M^{\op}$) meaning that $\rho(ab)=\rho(b)\rho(a)$. We will often write $x \xi y = \lambda(x)\rho(y) \xi$ for all $x, y \in M$ and all $\xi \in \rL^2(M)$. The vector $J \xi$ will be also simply denoted by $\xi^{*}$ so that $(x\xi)^{*}=\xi^{*}x^{*}$. For every $\xi \in \rL^2(M)$ we denote by $|\xi| \in \rL^2(M)_+$ its positive part. If $\xi \in \rL^2(M)$ we denote by $\xi \xi^* \in M_*^+$ the positive functional on $M$ defined by $x \mapsto \langle x \xi, \xi \rangle$. For any positive functional $\varphi \in M_\ast^+$, there exists a unique $\xi \in \rL^2(M)_+$ such that $\xi^2=\varphi$. We denote it by $\varphi^{1/2} \in \rL^2(M)_+$. We then have $\|x\|_\varphi = \|x \varphi^{1/2}\|$ for all $x \in M$.

For every positive functional $\varphi \in M_\ast$, we have 
$$\left\{\eta \in \rL^2(M) : |\eta|^2 \leq \varphi \right\}  =  \Ball(M) \varphi^{1/2}$$ (see the discussion after \cite[Lemma 3.2]{Ma16} for further details). Moreover, by the polarization identity, we have that the set of \emph{$\varphi$-bounded} vectors
$$\left\{\eta \in \rL^2(M) : \exists \lambda \in \R_+, \;  |\eta|^2 \leq \lambda\varphi  \text{ and }  |\eta^{*}|^2 \leq \lambda\varphi  \right\} $$ is linearly spanned by $\left\{\eta \in \rL^2(M)_+ : \eta^2 \leq \varphi \right\}$.

Finally, we recall that for any pair of faithful normal states $\varphi_1,\varphi_2$ on $M$, the \emph{relative modular operator} $\Delta_{\varphi_1,\varphi_2}$ is the unique closed unbounded positive operator on $\rL^2(M)$ such that the graph of $\Delta_{\varphi_1,\varphi_2}^{1/2}$ is the closure of $\{ (x\varphi_2^{1/2},\varphi_1^{1/2}x) \mid x \in M \} \subset \rL^2(M) \oplus \rL^2(M)$. When $\varphi_1=\varphi_2=\varphi$, we simply denote it by $\Delta_\varphi$.

\subsection*{One-parameter groups} A strongly continuous one-parameter group of isometries on a Banach space $E$ is a morphism $\sigma : \R \rightarrow \mathrm{Isom}(E)$ such that $\lim_{t \to 0} \|\sigma_t(x)-x\|=0$ for all $x \in E$. We denote by $S(\R)$ the set of all probability measures on $\R$ and for all $\mu \in S(\R)$, we define
$$\forall x \in E, \quad \sigma_\mu(x)=\int_{t \in \R} \sigma_t(x) \rd \mu (t)$$
where we use the usual Bochner integral for functions with values in a Banach space. Then for any $\varphi \in E^*$, we have 
$$\varphi(\sigma_\mu(x))=\int_{t \in \R} \varphi(\sigma_t(x)) \rd \mu (t).$$
When $f \in \rL^1(\R)^+$ with $\|f\|_1=1$, we will also use the notation $\sigma_f$ where we view $f$ as a probability measure on $\R$.

Now, let $M$ be a von Neumann algebra. A \emph{flow} on $M$ is a morphism $\sigma : \R \rightarrow \Aut(M)$ which is continuous for the $u$-topology. This means that the induced action $\sigma_* : \R \curvearrowright M_*$ is a strongly continuous one-parameter group of isometries of the Banach space $M_*$. Then, for all $\mu \in S(\R)$, we define $\sigma_\mu : M \rightarrow M$ as the dual map of $(\sigma_*)_{\mu}:M_* \rightarrow M_*$. In this context, we will also use the notation 
$$\forall x \in M, \quad \sigma_\mu(x) = \int_{t \in \R} \sigma_t(x) \rd \mu (t)$$
but we emphasize the fact that this is \emph{not} a Bochner integral with values in the Banach space $M$. For this reason, the relation
$$\varphi(\sigma_\mu(x))=\int_{t \in \R} \varphi(\sigma_t(x)) \rd \mu (t),$$
which holds by definition if $\varphi \in M_*$, \emph{fails} in general if we only have $\varphi \in M^*$. However, if the function $t \mapsto \sigma_t(x) \in M$ is norm continuous then the expression above for $\sigma_\mu(x)$ becomes a true Bochner integral and the formula for $\varphi(\sigma_\mu(x))$ holds even when $\varphi \in M^*$. In that respect, it is important to recall that $\mathcal{A}= \{ x \in M \mid t \mapsto \sigma_t(x) \text{ is norm continuous} \}$ is a strongly dense $C^*$-subalgebra of $M$. In fact, for any $f \in \rL^1(\R)$ and every $x \in M$, we have $\sigma_f(x) \in \mathcal{A}$.

\section{Strongly invariant states and approximate eigenvectors}

\begin{df}
Let $H$ be a Hilbert space and let $X$ be a self-adjoint operator on $H$. We say that a state $\Phi \in \B(H)^*$ is an \emph{approximate eigenstate} of $X$ if there exists a net of vectors $(\xi_i)_{i \in I}$ in the domain of $X$ such that $\lim_i \langle T \xi_i, \xi_i \rangle=\Phi(T) $ for all $T \in \B(H)$ and 
$$\lim_i \inf_{\lambda \in \R} \| (X-\lambda)\xi_i \|=0.$$
\end{df}

We denote by $\mathcal{E}(X)$ the set of all approximate eigenstates of $X$. The set $\mathcal{E}(X)$ is closed, hence compact, for the weak$^*$ topology. Indeed, we have
$$\mathcal{E}(X)=\bigcap_{\varepsilon > 0} \overline{\{ \langle \bdot \xi, \xi \rangle \mid \xi \in \dom(X), \; \inf_{\lambda \in \R} \| (X-\lambda)\xi \| \leq \varepsilon \} }^{\mathrm{w}^*}.$$
We denote by $\conv \mathcal{E}(X)$ the weak$^*$ closed convex hull of $\mathcal{E}(X)$. Since $\mathcal{E}(X)$ is compact, for any state $\Psi \in \B(H)^*$, we have $ \Psi \in \conv \mathcal{E}(X)$ if and only if $\Psi$ is the barycenter of some probability measure $\mu$ on $\mathcal{E}(X)$, i.e.\
$$\forall T \in \B(H), \quad \Psi(T)=\int_{\psi \in \mathcal{E}(X)} \psi(T) \rd \mu(\psi).$$
Intuitively, this means that $\Psi$ can be desintegrated into approximate eigenstates of $X$. The main theorem of this section provides a characterization of those states which admit such a nice decomposition.
\begin{thm} \label{thm state strong invariant}
Let $H$ be a Hilbert space and let $X$ be a self-adjoint operator on $H$. Let $\sigma : \R \curvearrowright \B(H)$ be the associated flow defined by $\sigma_t = \Ad(e^{\ri t X})$ for all $t \in \R$. Then for any state $\Psi \in \B(H)^*$, the following are equivalent:
\begin{itemize}
\item [$(\rm i)$] $\Psi \circ \sigma_\mu = \Psi$ for every probability measure $\mu$ on $\R$.
\item [$(\rm ii)$] $\Psi  \in \conv \mathcal{E}(X)$.
\end{itemize}
\end{thm}
A state $\Psi$ satisfying condition $(\rm i)$ will be called \emph{strongly $\sigma$-invariant}. A strongly $\sigma$-invariant state is of course $\sigma$-invariant, but the converse is not true in general when $\Psi$ is not normal because the map $t \mapsto \sigma_t(T)$ is not necessarily norm continous so that we do \emph{not} have
$$ \Psi \left(\int_{t \in \R} \sigma_t(T) \rd \mu (t) \right)=\int_{t \in \R} \Psi(\sigma_t(T)) \rd \mu (t).$$
See the preliminary section for more details.

We will need several lemmas in order to prove Theorem \ref{thm state strong invariant}.

\begin{lem} \label{lem strongly invariant}
Let $E$ be a Banach space and let $\alpha : \R \curvearrowright E$ be a strongly continous one-parameter group of isometries. Let $A$ be the infinitesimal generator of $\alpha$. For any bounded net $(x_i)_{i \in I}$ in $E$, the following are equivalent:
\begin{itemize}
\item [$(\rm i)$] For all $T > 0$, $\lim_{i \to \infty} \sup_{t \in [-T,T]} \| \alpha_t(x_i)-x_i \| = 0$.
\item [$(\rm ii)$] For all $\mu \in S(\R)$, $\lim_{i } \| \alpha_\mu(x_i)-x_i \|=0$.
\item [$(\rm iii)$] There is a net $(y_i)_{i \in I}$ in the domain of $A$ such that $\lim_i \| y_i-x_i\|=0$ and $\lim_i \| Ay_i \|=0$.
\end{itemize}
In that case, we say that the net $(x_i)_{i \in I}$ is strongly $\alpha$-invariant.
\end{lem}
\begin{proof}
$(\rm i) \Rightarrow (\rm ii)$ is obvious. Now, assume $(\rm ii)$ and let us show $(\rm iii)$. Take any $f \in \rL^1(\R)^+$ with $\|f\|_1=1$ whith Fourier transform $\widehat{f}$ supported on $[-1,1]$ and let $y_i=\alpha_f(x_i)$ for all $i \in  I$. Then $y_i$ is in the domain of $A$ and we have $\lim_i \|y_i-x_i \|=0$ by assumption. Let us show that $\lim_i \|Ay_i\|=0$. For $\varepsilon \in ]0,1[$, choose $g \in \rL^1(\R)^+, \: \|g\|_1=1$ with $\widehat{g}$ supported on $[-\varepsilon, \varepsilon]$. Then for all $i \in I$, we have $\|A\alpha_{g}(y_i)\| \leq \varepsilon \|y_i\|\leq M\varepsilon$ where $M=\sup_i \|y_i\|$. Moreover, we have $\limsup_i \| y_i-\alpha_g(y_i) \| \leq \limsup_i \| x_i-\alpha_{g}(x_i) \|=0$. Therefore we have $\limsup_i \| Ay_i-A\alpha_g(y_i)\| \leq \|y_i-\alpha_{g}(y_i)\|=0$. Since $\| A\alpha_{g}(y_i)\| \leq \varepsilon$, we conclude that $\limsup_i \|Ay_i\| \leq M\varepsilon$. This holds for all $\varepsilon > 0$, thus $\lim_i \|Ay_i\|=0$.
Finally, let us show that $(\rm iii) \Rightarrow (\rm i)$. For any $T > 0$, there exists a constant $\kappa > 0$ such that $|e^{\ri t x}-1| \leq \kappa |x|$ for all $t \in [-T,T]$. Hence, for all $y$ in the domain of $A$, we have $ \|\alpha_t(y)-y \| \leq \kappa \|A y\|$ and we conclude easily that $(\rm i)$ holds.
\end{proof}

\begin{lem} \label{lem state net}
Let $H$ be a Hilbert space and let $\sigma : \R \curvearrowright \B(H)$ be a flow. Then a state $\Psi \in \B(H)^*$ is strongly $\sigma$-invariant if and only if there exists a strongly $\sigma$-invariant net $(\Psi_i)_{i \in I}$ of normal states on $\B(H)$ such that $\lim_i \Psi_i=\Psi$ in the weak$^*$ topology.
\end{lem}
\begin{proof} The if direction follows directly from item $(\rm ii)$ in Lemma \ref{lem strongly invariant}. Let us prove the only if direction. Suppose that $\Psi$ is strongly $\sigma$-invariant. Let $I$ be the directed set of all triples $(K,F,\varepsilon)$ where $K \subset \B(H)$ and $F \subset S(\R)$ are finite subsets and $\varepsilon > 0$ is a positive real number. Fix $i=(K,F,\varepsilon) \in I$. Define a subset of $(\B(H)^*)^F$ by
$$ W=\{ ( \psi - \psi \circ \sigma_\mu )_{\mu \in F} \in (\B(H)_*)^F \mid \psi \in \B(H)_*^+, \: \psi(1)=1 \text{ and } \forall T \in K, \: |\psi(T)-\Psi(T)| \leq \varepsilon \}.$$
By assumption, we have that $0=(\Psi-\Psi \circ \sigma_\mu)_{\mu \in F}$ belongs the the weak$^*$ closure of $W$. Hence $0$ belongs to the weak closure of $W$ in $(\B(H)_*)^F$. Since $W$ is convex, then $0$ also belongs to the norm closure of $W$ by the Hahn-Banach theorem. This means that we can find a state $\Psi_i \in \B(H)_*$ such that $\| \Psi_i-\Psi_i \circ \sigma_\mu \| \leq \varepsilon$ for all $\mu \in F$ and $|\Psi_i(T)-\Psi(T)| \leq \varepsilon$ for all $T \in K$. Therefore, we have a constructed a net of normal states $(\Psi_i)_{i \in I}$ which is strongly $\sigma$-invariant and such that $\lim \Psi_i=\Psi$ in the weak$^*$ topology.
\end{proof}

\begin{lem} \label{lem inequality}
Let $H$ be a Hilbert space and let $X$ be a self-adjoint operator on $H$. For any bounded self-adjoint operator $T \in \B(H)$ the following are equivalent:
\begin{enumerate}
\item [$(\rm i)$] $\Phi(T) \leq 0$ for all $\Phi \in \mathcal{E}(X)$.
\item [$(\rm ii)$] For every $\varepsilon > 0$, there exists a constant $\kappa > 0$ such that for every $\lambda \in \R$, we have
$$T \leq \kappa |X - \lambda|^2 +\varepsilon.$$
\item [$(\rm iii)$] For every $\varepsilon > 0$, there exists a constant $\kappa > 0$ such that for any self-adjoint operator $Y$ on any Hilbert space $K$, we have
$$T \otimes 1 \leq \kappa |X \otimes 1 - 1 \otimes Y|^2 +\varepsilon$$
where the operators $T \otimes 1$, $X \otimes 1$ and $1 \otimes Y$ act on $H \otimes K$.
\end{enumerate}
\end{lem}
\begin{proof}
$(\rm i) \Rightarrow (\rm ii)$. Suppose that $\Phi(T) \leq 0$ for all $\Phi \in \mathcal{E}(X)$. Take $\varepsilon > 0$. Let us show that there exists a constant $\kappa > 0$ such that $T \leq \kappa |X-\lambda|^2+\varepsilon$ for every $\lambda \in \R$. If not, then for every $n \in \N$, we can find a unit vector $\xi_n$ in the domain of $X$ and some $\lambda_n \in \R$ such that $\langle T \xi_n, \xi_n \rangle \geq n \|(X-\lambda_n)\xi_n \|^2 + \varepsilon$. In particular, since $T$ is bounded, we have $\lim_n \|(X-\lambda_n) \xi_n \|=0$. Take $\omega \in \beta \N \setminus \N$ and define a state $\Phi$ on $\B(H)$ by the weak$^*$-limit $\Phi=\lim_{n \to \omega} \langle \bdot \xi_n , \xi_n \rangle$. Then we have $\Phi \in \mathcal{E}(X)$ and $\Phi(T) \geq \varepsilon$ but this is not possible by assumption. Hence there must exist some $\kappa > 0$ such that $T \leq \kappa |X- \lambda|^2 + \varepsilon$.

$(\rm ii) \Rightarrow (\rm i)$. Fix $\varepsilon > 0$ and take $\kappa > 0$ as in $(\rm ii)$. Then for every unit vector $\xi$ in the domain of $X$, we have $\langle T \xi, \xi \rangle \leq \kappa \inf_{\lambda \in \R} \| (X-\lambda) \xi \|^2+\varepsilon$. This implies that $\Phi(T) \leq \varepsilon$ for every $\Phi \in \mathcal{E}(X)$. Since this holds for every $\varepsilon > 0$, we conclude that $\Phi(T) \leq 0$ for all $\Phi \in \mathcal{E}(X)$.

Finally, the implication $(\rm iii) \Rightarrow (\rm ii)$ is obtained by taking $Y=\lambda$ on $K=\C$, while the implication $(\rm ii) \Rightarrow (\rm iii)$ follows from the spectral theorem applied to $Y$. See, for example, the proof of \cite[Lemma 4.1]{HMV16}.
\end{proof}

\begin{proof}[Proof of Theorem \ref{thm state strong invariant}]
$(\rm i) \Rightarrow (\rm ii)$. Since $\Psi$ is strongly $\sigma$-invariant, then, by Lemma \ref{lem state net}, we can find a strongly $\sigma$-invariant net $(\Psi_i)_{i \in I}$ of normal states on $\B(H)$ such that $\lim_i \Psi_i=\Psi$ in the weak$^*$ topology. Let $\Psi_i^{1/2} \in \mathrm{HS}(H)$ be the Hilbert-Schmidt operator associated to $\Psi_i$ for every $i \in I$. Note that the flow $\sigma$ restricts to a one-parameter unitary group $U=(U_t)_{t \in \R}$ on the Hilbert space $\mathrm{HS}(H)$. Under the identification $\mathrm{HS}(H) \cong H \otimes \overline{H}$, we have the formula $U_t = e^{\ri t X} \otimes \overline{e^{\ri t X}}=e^{\ri t X} \otimes e^{-\ri t \overline{X}}$ and the infinitesimal generator of $(U_t)_{t \in \R}$ is thus given by $A=X \otimes 1 - 1 \otimes \overline{X}$. Since $(\Psi_i)_{i \in I}$ is strongly $\sigma$-invariant, then by the Powers-St\o rmer inequality, the net $(\Psi^{1/2}_i)_{i \in I}$ is also strongly $U$-invariant. Therefore, by Lemma \ref{lem strongly invariant}, we can find a net $(\xi_i)_{i \in I}$ in the domain of $A$ such that $\lim_i\| \Psi_i^{1/2}-\xi_i \|=0$ and $\lim_i \|A \xi_i \|=0$.

Now, suppose by contradiction that $\Psi \notin \conv \mathcal{E}(X)$. Then by the Hahn-Banach separation theorem, we can find a bounded self-adjoint operator $T \in \B(H)$ such that $\Psi(T) > 0$ and $\Phi(T) \leq 0$ for all $\Phi \in \mathcal{E}(X)$. Take $0 < \varepsilon < \Psi(T)$. By Lemma \ref{lem inequality}, we can find a constant $\kappa > 0$ such that $T \otimes 1 \leq \kappa |A|^2 + \varepsilon$ as operators on $H \otimes \overline{H}$. By applying $\langle \bdot \xi_i , \xi_i \rangle$ to this inequality, we get $\langle (T \otimes 1) \xi_i, \xi_i \rangle \leq \kappa \| A \xi_i \|^2+ \varepsilon\|\xi_i\|^2$ for all $i \in I$. Hence, we get
\begin{align*}
\Psi(T)	&= \lim_i \langle (T \otimes 1)\Psi_i^{1/2},\Psi_i^{1/2} \rangle \\
			&= \lim_i \langle (T \otimes 1)\xi_i,\xi_i \rangle \\
			&\leq  \lim_i \left( \kappa \| A \xi_i \|^2+\varepsilon \|\xi_i\|^2 \right) \\
			&=0+\varepsilon.
\end{align*}
From this contradiction, we conclude that $\Psi \in \conv \mathcal{E}(X)$.

$(\rm ii) \Rightarrow (\rm i)$. Note that the set of all strongly $\sigma$-invariant states is convex and closed in the weak$^*$ topology. Hence it is enough to show that every approximate eigenstate $\Phi \in \mathcal{E}(X)$ is strongly $\sigma$-invariant. Let $(\xi_i)_{i \in I}$ be a net of unit vectors in $\dom(X)$ such that $\Phi = \lim_i \langle \bdot \xi_i, \xi_i \rangle$ and $\lim_i \inf_{\lambda \in \R} \| (X-\lambda) \xi_i \|=0$. Fix $T > 0$. Then we can find a constant $\kappa > 0$ such that $|e^{\ri t x} - 1| \leq \kappa |x|$ for all $x \in \R$ and all $t \in [-T,T]$. By applying this inequality to the operator $X-\lambda$, we get $$ \forall \lambda \in \R, \; \forall t \in [-T,T], \; \forall \xi \in \dom(X), \quad \|(e^{\ri t X}-e^{\ri t \lambda})\xi\| \leq \kappa \|(X-\lambda)\xi\|.$$  Let $\Phi_i=\langle \bdot \xi_i, \xi_i \rangle$. Then, for all $t \in [-T,T]$ and all $\lambda \in \R$, we have
\begin{align*}
\| \sigma_t(\Phi_i)-\Phi_i \| &=\| \langle \bdot e^{\ri t X}\xi_i,e^{\ri t X}\xi_i \rangle -  \langle \bdot e^{\ri t \lambda}\xi_i,e^{\ri t\lambda}\xi_i \rangle  \| \\
 &\leq 2\|(e^{\ri t X}-e^{\ri t \lambda})\xi_i\| \\
 &\leq 2\kappa \|(X-\lambda)\xi_i\|.
 \end{align*}
Since $\lim_i \inf_{\lambda \in \R} \| (X-\lambda) \xi_i \|=0$ we get
$\lim_i \sup_{t \in [-T,T]} \| \sigma_t(\Phi_i)-\Phi_i \| = 0$. This shows that the net $(\Phi_i)_{i \in I}$ is strongly $\sigma$-invariant. We conclude that $\Phi=\lim_i \Phi_i$ is strongly $\sigma$-invariant by Lemma \ref{lem state net}.
\end{proof}

The definition of approximate eigenstates is highly sensitive to the \emph{uniform} topology of $\R$. We clarify this issue and explain the notion of \emph{asymptotic eigenvalue} of an approximate eigenstate.

\begin{lem} \label{lem function eigenstate}
Let $X$ be a self-adjoint operator on a Hilbert space $X$ and suppose we have a net $(\lambda_i)_{i \in I}$ in $\R$ and a bounded net $(\xi_i)_{i \in I}$ in the domain of $X$ such that $\lim_i \| (X-\lambda_i)\xi_i \|=0$. Then we have:
\begin{itemize}
\item [$(\rm i)$] $\lim_i \| f(X-\lambda_i) \xi_i \|=0$ for every bounded measurable function $f:\R \rightarrow \C$ such that $f$ is continuous at $0$ and $f(0)=0$.
\item [$(\rm ii)$] $\lim_i \| f(X)\xi_i-f(\lambda_i)\xi_i \|=0$ for every bounded uniformly continuous function $f : \R \rightarrow \C$.
\end{itemize}
\end{lem}
\begin{proof}
$(\rm i)$. Take $\varepsilon > 0$. Then there exists $\delta > 0$ such that $|f(x)| \leq \varepsilon$ for all $x \in [-\delta, \delta]$. This implies that $|f(x)| \leq \frac{1}{\delta}\|f\|_\infty |x| + \varepsilon$ for all $x \in \R$. By applying this inequality to $X-\lambda_i$, we get $\|f(X-\lambda_i) \xi_i \| \leq \frac{1}{\delta} \|f \|_\infty \| (X-\lambda_i) \xi_i \| + \varepsilon \|\xi_i\|$ for all $i \in I$. This shows that $\limsup_i \|f(X-\lambda_i) \xi_i \| \leq \varepsilon \sup_i \| \xi_i\|$. Since this holds for every $\varepsilon > 0$, we conclude that $\lim_i \| f(X-\lambda_i) \xi_i \|=0$.
$(\rm ii)$. Take $\varepsilon > 0$. Since $f$ is uniformly continuous, there exists $\delta > 0$ such that $|f(x)-f(y)| \leq \varepsilon$ for all $x,y \in \R$ with $|x-y| \leq \delta$. This implies that $|f(x)-f(\lambda_i)| \leq \frac{1}{\delta}\|f\|_\infty |x-\lambda_i|+\varepsilon$ for all $x \in \R$. Then, similarly to item $(\rm i)$, we conclude by applying this inequality to $X$.
\end{proof}

Let $C_u(\R)$ be the set all bounded uniformly continuous functions on $\R$. We define the compactification $\overline{\R}^u$ of $\R$  as the Gelfand spectrum of the commutative $C^*$-algebra $C_u(\R)$ (this is called the \emph{Samuel compactification} in the litterature). If $A \subset \R$ is a closed subset of $\R$, we denote by $\overline{A}^u$ its closure in $\overline{\R}^u$.

\begin{prop} \label{prop asymptotic eigenvalue}
Let $X$ be a self-adjoint operator on a Hilbert space $H$. Let $\Phi \in \B(H)^*$ be a state. Then $\Phi \in \mathcal{E}(X)$ if and only if there exists $\omega \in \overline{\spec(X)}^u$, a net $(\lambda_i)_{i \in I}$ in $\spec(X)$ and a net $(\xi_i)_{i \in I}$ in the domain of $X$ such that:
\begin{itemize}
\item [$(\rm i)$] $\Phi = \lim_i \langle \bdot \xi_i, \xi_i \rangle$.
\item [$(\rm ii)$] $\lambda_i \to \omega$ when $i \to \infty$.
\item [$(\rm iii)$] $\lim_i \| (X-\lambda_i) \xi_i \|=0$.
\end{itemize}
In that case, $\omega$ is unique and is characterized by $\Phi(f(X))=f(\omega)$ for all $f \in C_u(\R)$.
\end{prop}

\begin{df}
We call $\omega \in \overline{\spec(X)}^u$ the \emph{asymptotic eigenvalue} of $\Phi$ and we denote by $\mathcal{E}_{\omega}(X) \subset \mathcal{E}(X)$ the set of all approximate eigenstates of $X$ with asymptotic eigenvalue $\omega$.
\end{df}

\begin{proof}[Proof of Proposition \ref{prop asymptotic eigenvalue}]
The if direction follows from the definition of approximate eigenstates. Let us prove the other direction. Suppose that $\Phi \in \mathcal{E}(X)$. Then we can find a bounded net $(\xi_i)_{i \in I}$ in the domain of $X$ and a net $(\lambda_i)_{i \in I} \in \R$ such that $\Phi = \lim_i \langle \bdot \xi_i, \xi_i \rangle$ and $\lim_i \|(X-\lambda_i)\xi_i\|=0$. Then for every $\varepsilon > 0$, by applying Lemma \ref{lem function eigenstate}.$(\rm i)$ to $f=1-1_{[-\varepsilon, \varepsilon]}$, we have 
$$\lim_i \| \xi_i - 1_{[\lambda_i-\varepsilon,\lambda_i+\varepsilon]}(X)\xi_i\|=\lim_i \| \xi_i - 1_{[-\varepsilon,\varepsilon]}(X-\lambda_i)\xi_i\|=0.$$
In particular, for $i$ large enough we have $1_{[\lambda_i - \varepsilon, \lambda_i + \varepsilon]}(X) \neq 0$, i.e.\ $\spec(X) \cap [\lambda_i - \varepsilon, \lambda_i + \varepsilon] \neq \emptyset$. This holds for every $\varepsilon > 0$ so that the distance between $\lambda_i$ and $\spec(X)$ tends to $0$ when $i \to \infty$. Thus, up to replacing the net $(\lambda_i)_{i \in I}$ by a net $(\mu_i)_{i \in I}$ in $\spec(X)$ such that $\lim_i \| \lambda_i - \mu_i \|=0$, we can assume that $\lambda_i \in \spec(X)$ for all $i \in I$. Now, take $\omega$ an accumulation point of the net $(\lambda_i)_{i \in I}$ in $\overline{\spec(X)}^u$. Then, for every $f \in C_u(\R)$, $f(\omega)$ is an accumulation point of $(f(\lambda_i))_{i \in I}$. But, by Lemma \ref{lem function eigenstate}.$(\rm ii)$, we have $\lim_i f(\lambda_i)=\Phi(f(X))$, hence $\Phi(f(X))=f(\omega)$ for every $f \in C_u(\R)$. In particular, if $\omega'$ is another accumulation point of $(\lambda_i)_{i \in I}$, then we have $f(\omega')=\Phi(f(X))=f(\omega)$ for all $f \in C_u(\R)$ which means that $\omega'=\omega$. This shows that $\omega$ is the unique accumulation point of $(\lambda_i)_{i \in I}$ and $\lim_i \lambda_i=\omega$.
\end{proof}

In the applications to type $\III$ factors, we will be interested in approximate eigenstates of the logarithm of the modular operator $\Delta$. For every $\lambda \in \R^*_+$, it is easy to check that $\mathcal{E}_{\lambda}(\Delta)=\mathcal{E}_{\log \lambda}(\log \Delta)$. However, since the logarithm is not uniformly continuous, it is \emph{not true} in general that $\mathcal{E}(\Delta) = \mathcal{E}(\log \Delta)$. Instead, we give the following description.

\begin{prop} \label{prop log}
Let $H$ be a Hilbert space and let $\Delta$ be a positive definite operator on $H$. Then for any state $\Phi \in \B(H)^*$, the following are equivalent: 
\begin{itemize}
\item [$(\rm i)$] $\Phi \in \mathcal{E}( \log \Delta )$
\item [$(\rm ii)$] There exists a net of vectors $(\xi_i)_{i \in I}$ in the domain of $\Delta$ and a net $(\lambda_i)_{i \in I}$ in $\R^*_+$ such that $\Phi = \lim_i \langle \bdot \xi_i, \xi_i \rangle$ in the weak$^*$ topology and
$$ \lim_i  \| (\lambda_i^{-1} \Delta -1) \xi_i \|=0.$$
\item [$(\rm iii)$] There exists a net of vectors $(\xi_i)_{i \in I}$ which are analytic for $\Delta$ and a net $(\lambda_i)_{i \in I}$ in $\R^*_+$ such that $\Phi = \lim_i \langle \bdot \xi_i, \xi_i \rangle$ in the weak$^*$ topology and the net of analytic functions
$$f_i(z)= (\lambda_i^{-z} \Delta^z -1) \xi_i $$
converges to $0$ uniformly on all compact subsets of $\C$.
\end{itemize}
Moreover, if $\Phi \in \mathcal{E}_\omega( \log \Delta)$, we can choose the net $(\lambda_i)_{i \in I}$ such that $\lambda_i \in \spec(\Delta)$ for all $i \in I$ and $\lim_i \log (\lambda_i)=\omega$.
\end{prop}
\begin{proof}
$(\rm i) \Rightarrow (\rm iii)$. Take a net of unit vectors $(\xi_i)_{i \in I}$ in the domain of $X=\log \Delta$ and a net $(\lambda_i)_{i \in I}$ in $\R^*_+$ such that $\Phi = \lim_i \langle \bdot \xi_i, \xi_i \rangle$ in the weak$^*$ topology and $ \lim_i  \| (X-\log(\lambda_i)) \xi_i \|=0$. Let $\eta_i=1_{[\log(\lambda_i)-1, \log(\lambda_i)+1]}(X) \xi_i$. Then we have $\lim_i \|x_i-\eta_i \|=0$ and in particular $\Phi=\lim_i \langle \bdot \eta_i, \eta_i \rangle$. Moreover, $\eta_i$ is $\Delta$-analytic for all $i \in I$. Take $K$ a compact subset of $\C$. Then we can find a constant $\kappa > 0$ such $|e^{zx}-1| \leq \kappa |x|$ for all $z \in K$ and all $x \in [-1,1]$. This implies that for all $z \in K$, we have
$$ \| (\lambda_i^{-z} \Delta^{z} - 1)\eta_i \|= \| (e^{z(X-\log(\lambda_i))} - 1)\eta_i \| \leq \kappa \| (X-\log(\lambda_i)) \eta_i \| \to 0 \text{ when } i \to \infty$$
as we wanted.

The implication $(\rm iii) \Rightarrow (\rm ii)$ is obvious. For $(\rm ii) \Rightarrow (\rm i)$ we proceed in the same way, by letting $\eta_i = 1_{[\log \lambda_i-1, \log \lambda_i+1]}(X)\xi_i$ and using the fact that $|x| \leq \kappa |e^x-1|$ for all $x \in [-1,1]$ and some constant $\kappa > 0$.
\end{proof}

Finally, let us end this section with the following strong fixed point property which is the main tool we use to construct strongly invariant states.

\begin{prop} \label{prop strong fixed point}
Let $H$ be a Hilbert space and $\sigma: \R \curvearrowright \B(H)$ a flow. Let $K \subset S(\B(H))$ be a non-empty weak$^*$-closed set which is strongly $\sigma$-invariant, meaning that $\Psi \circ \sigma_\mu \in K$ for all $\Psi \in K$ and all $\mu \in S(\R)$. Then there exists a state $\Psi \in K$ which is strongly $\sigma$-invariant.
\end{prop}
\begin{proof}
Since $K$ is $\sigma$-invariant and $\R$ is amenable as a discrete group, we can find $\Psi \in K$ which is $\sigma$-invariant. Take any $f \in \rL^1(\R)^+$ with $\|f\|_1=1$ and let $\Psi'=\Psi \circ \sigma_f \in K$. We will show that $\Psi'$ is strongly $\sigma$-invariant.

Take $T \in \B(H)$ and let $T'=\sigma_{f}(T)$. Observe that the map $t \mapsto \sigma_t(T')$ is norm continuous because we have
$$ \| \sigma_t(T')-T' \| \leq \| f(\bdot + t)-f \|_1 \cdot \| T \|.$$
Therefore, for any $\mu \in S(\R)$, we have 
$$\Psi(\sigma_\mu(T'))=\int_{\R}  \Psi(\sigma_t(T')) \rd \mu(t)=\int_{\R} \Psi(T') \rd \mu(t)=\Psi(T').$$
This implies that $\Psi'$ is strongly $\sigma$-invariant because
$$ \Psi'(\sigma_\mu(T))=\Psi(\sigma_{f}( \sigma_\mu(T)))=\Psi(\sigma_{\mu}( \sigma_{f}(T)))=\Psi(\sigma_\mu(T'))=\Psi(T')=\Psi'(T).$$
\end{proof}

\section{Two lemmas}

In this section, we prove two lemmas that we will use in the proofs of the main theorems.

\begin{lem} \label{lem binormal invariant}
Let $M$ be a von Neumann algebra and $\sigma : \R \curvearrowright \B(\rL^2(M))$ a flow which leaves $\lambda(M)$ and $\rho(M)$ globally invariant. Let $\Phi$ be a state on $C^*_{\lambda \cdot \rho}(M)$ such that $\Phi|_{\lambda(M)}$ and $\Phi|_{\rho(M)}$ are both normal states. Suppose that $\Phi$ is $\sigma$-invariant. Then the set $K$ of all states $\Psi \in \B(\rL^2(M))^*$ extending $\Phi$ is strongly $\sigma$-invariant. In particular, there exists an extension $\Psi \in K$ which is strongly $\sigma$-invariant.
\end{lem}
\begin{proof}
Take $\Psi \in K$ and take $\mu$ a probability measure on $\R$. We have to show that $\Psi'=\Psi \circ \alpha_\mu \in K$, i.e\ that $\Psi'$ also extends $\Phi$. First, observe that since $\Phi$ is normal and $\sigma$-invariant on $\lambda(M)$ and $\rho(M)$, we have $\Psi'|_{\lambda(M)}=\Psi|_{\lambda(M)}=\Phi|_{\lambda(M)}$ and $\Psi'|_{\rho(M)}=\Psi|_{\rho(M)}=\Phi|_{\rho(M)}$. Let $\mathcal{A}=\{ x \in \lambda(M) \mid t \mapsto \sigma_t(x) \text{ is norm continuous} \}$ and $\mathcal{B}=\{ x \in \rho(M) \mid t \mapsto \sigma_t(x) \text{ is norm continuous} \}$. Then $\mathcal{A}$ and $\mathcal{B}$ are strongly dense $*$-subalgebra of $\lambda(M)$ and $\rho(M)$ respectively and we have $\Psi'(T)=\Psi(T)=\Phi(T)$ for all $T \in \mathcal{A}\cdot\mathcal{B}$. Pick $a \in \lambda(M)$ and $b \in \rho(M)$. Take two bounded nets $(a_i)_{i \in I}$ in $\mathcal{A}$ and $(b_i)_{i \in I}$ in $\mathcal{B}$ such that $a_i \to a$ and $b_i \to b$ $*$-strongly. We can assume that $\|a_i \| \leq \|a \|$ and $\|b_i \| \leq \|b \|$ for all $i \in I$. Since $\Psi'|_{\lambda(M)}=\Phi|_{\lambda(M)}$ and $\Psi'|_{\rho(M)}=\Phi|_{\rho(M)}$ are both normal states, we have 
$$ \lim_i \| ab-a_ib_i\|_{\Phi} \leq \lim_i \| a \| \|b-b_i\|_{\Phi}+\| b \| \|a-a_i\|_{\Phi} =0$$
and 
$$ \lim_i \| ab-a_ib_i\|_{\Psi'} \leq \lim_i \| a \| \|b-b_i\|_{\Psi'}+\| b \| \|a-a_i\|_{\Psi'} =0.$$
Therefore, we obtain $ \Phi(ab)=\lim_i \Phi(a_ib_i)$ and $\Psi'(ab)=\lim_i \Psi'(a_ib_i)$. But since $a_ib_i \in \mathcal{A} \cdot \mathcal{B}$, we have $\Psi'(a_ib_i)=\Phi(a_ib_i)$ for all $i \in I$, hence $\Psi'(ab)=\Phi(ab)$. Since this holds for every $a \in \lambda(M)$ and $b \in \rho(M)$, we conclude that $\Psi'(T)=\Phi(T)$ for all $T \in C^*_{\lambda \cdot \rho}(M)$, i.e.\ $\Psi' \in K$. This shows that $K$ is strongly $\sigma$-invariant and we conclude by Proposition \ref{prop strong fixed point}.
\end{proof}

The second lemma is intuitively easy to understand. It says that if $M \cong M \ovt R_\infty$, so that $M$ has central sequences with arbitrary asymptotic eigenvalue, then one can perturb any approximate eigenstate $\Psi$ in order to change its asymptotic eigenvalue \emph{without} changing the values that $\Psi$ takes on $C^*_{\lambda \cdot \rho}(M)$.

\begin{lem} \label{lem absorb}
Let $M$ be a von Neumann algebra with two faithful normal states $\varphi_1,\varphi_2$. Let $\Delta=\Delta_{\varphi_1,\varphi_2}$ and $\sigma_t=\Ad(\Delta^{\ri t})$ for all $t \in \R$. Suppose that $M \cong M \ovt R_\infty$. Then for any $\Psi \in \conv \mathcal{E}(\log \Delta)$, we can find $\Psi' \in \conv \mathcal{E}_0(\log \Delta)$ such that $\Psi'|_{C^*_{\lambda \cdot \rho}(M)}=\Psi|_{C^*_{\lambda \cdot \rho}(M)}$. 
\end{lem}
\begin{proof}
Observe that the set of all states $\Psi$ satisfying the conclusion of the lemma is convex and weak$^*$-closed. Hence, it is enough to show that it contains $\mathcal{E}(\log\Delta)$. So take $\Psi \in \mathcal{E}(\log\Delta)$. Let $\omega \in \overline{\R}^u$ be the asymptotic eigenvalue of $\Psi$. Define a directed set $J$ consisting of all pairs $(F,\varepsilon)$ where $F$ is a finite subset of $M$ and $\varepsilon$ is a positive real number. Let $\xi_1=\varphi_1^{1/2}$ and $\xi_2=\varphi_2^{1/2}$ and recall that the graph of $\Delta^{1/2}$ is the closure of $\{ (x\xi_2,\xi_1x) \mid x \in M\}$.

Suppose first that $\omega \geq 0$ (by this, we mean that $\omega \in \overline{\R}_+^u$). Fix $j=(F,\varepsilon) \in J$. Then, by Proposition \ref{prop log}, we can find $\lambda \geq 1$ and $x \in M$ such that $\|x\xi_2-\lambda^{-1} \xi_1x\| \leq \varepsilon$ and $|\Psi(T) - \langle T x\xi_2, x\xi_2 \rangle | \leq \varepsilon$ for all $T \in \lambda(F)\rho(F)$. Since $M \cong M \ovt R_\infty$, we can find a sequence of isometries $v_n \in M$ such that $\lim_n \|\lambda^{-1} v_n \xi_1-\xi_1 v_n\|=0$ and $\lim_n v_n^*av_n=a$ strongly for all $a \in F$. Then, we have $\lim_n \langle T  v_n x\xi_2,   v_n x \xi_2 \rangle = \langle T x\xi_2, x\xi_2 \rangle$ for all $T \in \lambda(F)\rho(F)$ and
$$\lim_n \| v_n x \xi_2-\xi_1v_n x \|=\lim_n \| v_n x \xi_2 -  \lambda^{-1} v_n \xi_1 x \| =\| x\xi_2 - \lambda^{-1} \xi_1x\| \leq \varepsilon.$$
Therefore, if we take $y_j=v_nx$ for $n$ large enough, we will have $\|y_j\xi_2-\xi_1y_j\| \leq 2\varepsilon$ and $|\Psi(T) - \langle T y_j\xi_2, y_j\xi_2 \rangle | \leq 2\varepsilon$ for all $T \in \lambda(F)\rho(F)$. Now take $\Phi \in \B(\rL^2(M))^*$ any accumulation point of the net of states $\langle \bdot y_j \xi_2, y_j \xi_2 \rangle, \: j \in J$. Then by construction, $\Phi \in \mathcal{E}_1(\Delta^{1/2})=\mathcal{E}_0(\log \Delta)$ and $\Phi|_{C^*_{\lambda \cdot \rho}(M)}=\Psi|_{C^*_{\lambda \cdot \rho}(M)}$.

Suppose now that $\omega \leq 0$. Fix $j=(F,\varepsilon) \in J$. Then, by Proposition \ref{prop log}, we can find $\lambda \leq 1$ and $x \in M$ such that $\|\xi_1 x-\lambda x \xi_2 \| \leq \varepsilon$ and $|\Psi(T) - \langle T \xi_1 x, \xi_1 x \rangle | \leq \varepsilon$ for all $T \in \lambda(F)\rho(F)$. Since $M \cong M \ovt R_\infty$, we can find a sequence of coisometries $v_n \in M$ such that $\lim_n \| v_n \xi_2-\lambda \xi_2 v_n\|=0$ and $\lim_n v_nav_n^*=a$ strongly for all $a \in F$. Then, we have $\lim_n \langle T  \xi_1 x v_n,    \xi_1x v_n \rangle = \langle T \xi_1x, \xi_1 x \rangle$ for all $T \in \lambda(F)\rho(F)$ and
$$\lim_n \| xv_n \xi_2-\xi_1 x v_n \|=\lim_n \| \lambda  x \xi_2 v_n -   \xi_1 xv_n \| =\| \lambda x\xi_2 - \xi_1x\| \leq \varepsilon.$$
Therefore, if we take $y_j=xv_n$ for $n$ large enough, we will have $\|y_j\xi_2-\xi_1y_j\| \leq 2\varepsilon$ and $|\Psi(T) - \langle T \xi_1y_j, \xi_1 y_j \rangle | \leq 2\varepsilon$ for all $T \in \lambda(F)\rho(F)$. Now take $\Phi \in \B(\rL^2(M))^*$ any accumulation point of the net of states $\langle \bdot \xi_1y_j, \xi_1 y_j \rangle, \: j \in J$. Then by construction, $\Phi \in \mathcal{E}_1(\Delta^{-{1/2}})=\mathcal{E}_0(\log \Delta)$ and $\Phi|_{C^*_{\lambda \cdot \rho}(M)}=\Psi|_{C^*_{\lambda \cdot \rho}(M)}$.

\end{proof}
\begin{rem} \label{type lambda}
In Lemma \ref{lem absorb}, if we know that $\mathrm{spec}(\Delta) \subset \lambda^\Z$, then the same conclusion holds if we only assume that $M \cong M \ovt R_\lambda$.
\end{rem}

\section{Full factors}

The following lemma is certainly well-known but we provide a proof for the reader's convenience.

\begin{lem} \label{lem compact}
Let $H$ be a Hilbert space and let $A \subset \B(H)$ be a $C^*$-algebra which is irreducible, i.e.\ $A'=\C$. Then, either $\K(H) \subset A$ or $\K(H) \cap A =\{0\}$. Moreover, in the latter case, every state on $A$ can be extended to a state on $\B(H)$ which vanishes on $\K(H)$.
\end{lem}
\begin{proof}
Suppose that $\K(H) \cap A \neq \{0\}$. Take $T \in \K(H) \cap A$ a non-zero self-adjoint operator. Since $T$ is compact, by taking a spectral projection of $T$, we obtain a non-zero finite rank projection $P \in \K(H) \cap A$. Take $p \leq P$ a rank one projection. Let $q \in \K(H)$ be any rank one projection. We can find a partial isometry $v \in \K(H)$ such that $v^*v=p$ and $vv^*=q$. Since $A''=\B(H)$, we can find a bounded net $(a_i)_{i \in I}$ in $A$ such that $a_i \to v$ in the $*$-strong topology. Then, since $P$ has finite rank, the net $(a_iP)_{i \in I}$ converges in norm to $vP=q$. Hence $q \in A$. This holds for every rank one projection $q$. Thus $\K(H) \subset A$. 

For the second part, suppose that $\K(H) \cap A =\{0\}$. Let $\pi : \B(H) \rightarrow \B(H)/\K(H)$ be the quotient map. Then the restriction of $\pi$ to $A$ is injective. Thus for any state $\varphi \in A^*$, we can find a state $\phi \in \pi(A)^*$ such that $\phi \circ \pi|_A = \varphi$. Now, by the Hahn-Banach theorem, we can find a state $\Phi$ on $\B(H)/\K(H)$ such that $\Phi |_{\pi(A)}=\phi$. Then $\Phi \circ \pi$ is a state on $\B(H)$ which extends $\varphi$ and vanishes on $\K(H)$.
\end{proof}

\begin{proof}[Proof of Theorem \ref{main full compact}]
The result is trivial if $M$ is of type $\mathrm{I}$ and is already knwon if $M$ is of type $\II$ by \cite[Theorem 2.1]{Co75b}. So we can assume that $M$ is a type $\III$ factor. By \cite[Theorem 3.2]{HMV16}, there exists a faithful normal state $\varphi$ on $M$, a finite set $F \subset M$ with $a \varphi^{1/2}=\varphi^{1/2}a^*$ for all $a \in F$, and a constant $\kappa > 0$ such that
$$\forall x \in M, \quad \|x-\varphi(x)\|_\varphi \leq \kappa \left( \sum_{a \in F} \|xa-ax\|_\varphi + \inf_{\lambda \in \R^*_+} \| x\varphi^{1/2}-\lambda \varphi^{1/2}x\| \right) .$$
Let $\sigma_t=\Ad(\Delta_\varphi^{\ri t})$. Let $\omega_\varphi$ be the state on $C^*_{\lambda \cdot \rho}(M)$ defined by $\omega_\varphi(T)=\langle T \varphi^{1/2}, \varphi^{1/2} \rangle$ for all $T \in C^*_{\lambda \cdot \rho}(M)$. Observe that $\omega_\varphi$ is $\sigma$-invariant. Let $e_\varphi$ be the rank one projection on $\varphi^{1/2}$. Suppose that $\K(\rL^2(M))$ is not contained in $C^*_{\lambda \cdot \rho}(M)$. Then, by Lemma \ref{lem compact}, the set $K$ of all states $\Psi$ on $\B(\rL^2(M))$ extending $\omega_\varphi$ and which satisfy $\Psi(e_\varphi)=0$ is not empty. Thanks to Lemma \ref{lem binormal invariant}, we know that $K$ is strongly $\sigma$-invariant. Hence, by Proposition \ref{prop strong fixed point}, we can find a state $\Psi \in K$ which is strongly $\sigma$-invariant. Now, by Theorem \ref{thm state strong invariant}, $\Psi$ is the barycenter of a probability measure $\mu$ on $\mathcal{E}(\log \Delta_\varphi)$. Pick $\psi$ in the support of $\mu$. Since $\Psi(e_\varphi)=0$ and $\Psi(|\lambda(a)-\rho(a^*)|^2)=\omega_\varphi(|\lambda(a)-\rho(a^*)|^2)=0$ for all $a \in F$, we must also have $\psi(e_\varphi)=0$ and $\psi(|\lambda(a)-\rho(a^*)|^2)=0$ for all $a \in F$. Since $\psi \in \mathcal{E}(\log\Delta_\varphi)$ and since the graph of $\Delta_\varphi^{1/2}$ is the closure of $\{ (x\varphi^{1/2},\varphi^{1/2}x) \mid x \in M\}$, then by Proposition \ref{prop log}, we can find a net $(x_i)_{i \in I}$ in $M$ with $\|x_i\|_\varphi=1$ for all $i \in I$ such that $\psi = \lim_i \langle \bdot x_i \varphi^{1/2},x_i \varphi^{1/2} \rangle$ in the weak$^*$ topology and $\lim_i \inf_{\lambda \in \R^*_+} \| x_i \varphi^{1/2}-\lambda^{-1} \varphi^{1/2}x_i \|=0$. Since $\psi(e_\varphi)=0$ we have $\lim_i \varphi(x_i)=0$ and since $\psi(|\lambda(a)-\rho(a^*)|^2)=0$, we have $\lim_i \| ax_i-x_ia \|_\varphi^2=0$ for all $a \in F$. This is a contradiction. Thus, we must have $\K(\rL^2(M)) \subset C^*_{\lambda \cdot \rho}(M)$.
\end{proof}

\begin{proof}[Proof of Corollary \ref{cor central net}]
Let $(x_i)_{i \in I}$ be a central bounded net in $M$. Then $\lambda(x_i)T-T\lambda(x_i) \to 0$ strongly for all $T \in C^*_{\lambda \cdot \rho}(M)$. Hence, by Theorem \ref{main full compact}, we have $\lambda(x_i)p-p\lambda(x_i) \to 0$ strongly where $p$ is the rank one projection on $\varphi^{1/2}$ for some faithful normal state $\varphi$. This means that $\lim_i \|x_i - \varphi(x_i)\|_\varphi=0$.
\end{proof}

\begin{proof}[Proof of Corollary \ref{cor topology}]
Let $(\theta_i)_{i \in I}$ be a net in $\Aut(M)$ such that $\theta_i(x) \to x$ strongly for all $x \in M$. Then $U_{\theta_i}T U_{\theta_i}^* \to T$ strongly for all $T \in C^*_{\lambda \cdot \rho}(M)$ where $U_{\theta_i}$ is the unitary implementation of $\theta_i$ on $\rL^2(M)$. Hence, by Theorem \ref{main full compact}, we have $U_{\theta_i}e_\varphi U_{\theta_i}^* \to e_\varphi$ strongly where $e_\varphi$ is the rank one projection on $\varphi^{1/2}$ for any $\varphi \in M_*$. This means that $\lim_i \| \theta_i(\varphi)^{1/2} -\varphi^{1/2} \|=\| U_{\theta_i} \varphi^{1/2} -\varphi^{1/2} \|=0$ for all $\varphi \in M_*^+$. By Araki-Powers-St\o rmer's inequality, we conclude that $\lim_i \|\theta_i(\varphi)-\varphi \|=0$ for all $\varphi \in M_*^+$, hence for all $\varphi \in M_*$.
\end{proof}

\section{Bicentralizer flow}

\begin{proof}[Proof of Theorem \ref{main bicentralizer}]
First we prove that if $M$ is a type $\III_1$ factor such that $M \cong M \ovt R_\infty$, then $M$ has trivial bicentralizer. We use Haagerup's criterion for triviality of the bicentralizer stated in \cite[Theorem 7.2]{AHHM18}. Let $\varphi$ be a faithful normal state on $M$. By minimality of the spatial tensor product, there exists a state $\varphi \otimes \varphi$ on $C^*_{\lambda \cdot \rho}(M)$ such that $(\varphi \otimes \varphi)(\lambda(a)\rho(b))=\varphi(a)\varphi(b)$ for all $a,b \in M$. By Lemma \ref{lem binormal invariant}, we can find a strongly $\sigma$-invariant state $\Psi$ on $\B(\rL^2(M))$ which extends $\varphi \otimes \varphi$. Then by Theorem \ref{thm state strong invariant} and Lemma \ref{lem absorb}, we can assume that $\Psi$ is the barycenter of some probability measure $\mu$ on $\mathcal{E}_0(\log\Delta_\varphi)$. Let $x \in M$ such that $x\varphi^{1/2}=\varphi^{1/2}x^*$, $\|x\|_\varphi=1$ and $\varphi(x)=0$. Since $\Psi(|\lambda(x)-\rho(x^*)|^2)=(\varphi \otimes \varphi)(|\lambda(x)-\rho(x^*)|^2)= 2\|x\|_\varphi^2=2$ and $\Psi(|\rho(x^*)|^2)=\|x\|_\varphi^2=1$, we can find $\psi$ in the support of $\mu$ such that $\psi(|\lambda(x)-\rho(x^*)|^2) \geq \frac{1}{2}  + \psi(|\rho(x^*)|^2)$. Since $\psi \in \mathcal{E}_0(\log\Delta_\varphi)$, we can find a net $(a_i)_{i \in I}$ in $M$ such that $\|a_i\|_\varphi=1$ for all $i \in I$, $\lim_i \|a_i \varphi^{1/2}-\varphi^{1/2} a_i\|=0$ and $\psi = \lim_i \langle \bdot a_i \varphi^{1/2},a_i \varphi^{1/2} \rangle$ in the weak$^*$ topology. Then we have 
$$\lim_i \| x a_i - a_i x \|_{\varphi}^2=\psi(|\lambda(x)-\rho(x^*)|^2) \geq \frac{1}{2}  + \psi(|\rho(x^*)|^2)=\frac{1}{2} + \lim_i \|a_ix\|_\varphi^2.$$ Therefore, for any $\delta > 0$, if we let $a=a_i$ with $i$ large enough, we will have
$$ \|a\|_\varphi + \|ax\|_\varphi < 3 \| xa-ax\|_\varphi$$
$$ \|a \varphi^{1/2}-\varphi^{1/2}a\| < \delta \|xa-ax\|_\varphi$$
which is exactly the criterion of \cite[Theorem 7.2]{AHHM18}. We conclude that $M$ has trivial bicentralizer.

Now, back to the general case, let $M$ be any type $\III_1$ factor with a faithful normal state $\varphi$. Let us show that the bicentralizer flow $\beta^\varphi : \R^*_+ \curvearrowright \B(M,\varphi)$ is ergodic. Suppose that the fixed point algebra $\B(M,\varphi)^{\beta^\varphi}$ is non-trivial. Then $\B(M,\varphi)^{\beta^\varphi}$ is a self-bicentralizing type $\III_1$ factor with trivial bicentralizer flow, hence $\B(M,\varphi)^{\beta^\varphi} \cong \B(M,\varphi)^{\beta^\varphi} \ovt R_\infty$ by \cite[Theorem B]{AHHM18}. But this is not possible by the first part of the proof. Therefore $\B(M,\varphi)^{\beta^\varphi}$ must be trivial and $\beta^\varphi$ is ergodic. Let us show that $\beta^\varphi_\lambda$ is ergodic for every $\lambda \in \R^*_+ \setminus \{1\}$. If not, then the fixed point algebra $\B(M,\varphi)^{\beta^\varphi_\lambda}$ is a self-bicentralizing type $\III_1$ factor with periodic bicentralizer flow. But this is not possible because a periodic flow on a non-amenable factor cannot be ergodic. Therefore $\B(M,\varphi)^{\beta^\varphi_\lambda}$ must be trivial and $\beta^\varphi_\lambda$ is ergodic. 

Finally, item $(\rm i)$ follows from \cite[Theorem C]{AHHM18} and item $(\rm ii)$ follows from \cite[Theorem B.$(\rm iii)$]{AHHM18}.
\end{proof}

\begin{prop} \label{prop bic approx}
Let $M$ be a type $\III_1$ factor. Then for any $\lambda \in ]0,1[$, the automorphism $\beta_\lambda \otimes \id \curvearrowright \B(M) \ovt R_\lambda$ is approximately inner.
\end{prop}
\begin{proof}
Let $\varphi$ be a faithful normal state on $M$ and $\psi$ a faithful normal state on $R_\lambda$. Take any free ultrafilter $\omega \in \beta \N \setminus \N$. Let $v \in \B(M,\varphi)^\omega$ be a non-zero partial isometry such that $v\varphi^\omega=\lambda \varphi^{\omega}v$ and let $w \in R_\lambda' \cap R_\lambda^\omega$ be a non-zero partial isometry such that $w\psi^{\omega}=\lambda^{-1}\psi^{\omega}w$. Define a non-zero partial isometry $u=v \otimes w \in (\B(M,\varphi) \ovt R_\lambda)^\omega$. Then for any $x \in \B(M,\varphi) \ovt R_\lambda$ we have $(\beta^{\varphi}_\lambda \otimes \id)(x)u=ux$. Moreover, we have $u(\varphi \otimes \psi)^\omega=(\varphi \otimes \psi)^\omega u$.This implies that $\beta^{\varphi}_\lambda \otimes \id$ is approximately inner (see the proof of \cite[Theorem 1]{Co85}).
\end{proof}

\section{Approximately inner automorphisms}

\begin{proof}[Proof of Theorem \ref{main weakly inner}]
Let $\theta$ be a weakly inner automorphism of $M$. In order to show that $\theta$ is approximately inner, we will use the criterion of \cite[Theorem III.1]{Co85}. Take $\xi_1, \dots, \xi_n \in \rL^2(M)$. We have to show that for every $\varepsilon > 0$, there exists a non-zero $x \in M$ such that  $\sum_k \| x \xi_k-\theta(\xi_k) x \|^2 \leq \varepsilon \sum_k \| x \xi_k\|^2$. Let $\varphi_1$ be a faithful normal state on $M$ such that every $\xi_k$ is $\varphi_1$-bounded. Then we have $\xi_k=a_k \varphi_1^{1/2}=\varphi_1^{1/2}b_k$ for some $a_k,b_k \in M$. Let $\varphi_2=\theta(\varphi_1)=\varphi_1 \circ \theta^{-1}$. Since $\theta$ is weakly inner, there exists an automorphism $\alpha$ of $C^*_{\lambda \cdot \rho}(M)$ such that $\alpha(\lambda(a)\rho(b))=\lambda(\theta(a))\rho(b)$ for all $a,b \in M$.  Define a state $\omega_\theta$ on $C^*_{\lambda \cdot \rho}(M)$ by $\omega_\theta(T)=\langle \alpha^{-1}(T) \varphi_1^{1/2}, \varphi_1^{1/2} \rangle$. Let $\Delta=\Delta_{\varphi_2, \varphi_1}$ and $\sigma_t=\Ad(\Delta^{\ri t})$ for all $t \in \R$. Observe that $\omega_\theta$ is $\sigma$-invariant. By Lemma \ref{lem binormal invariant}, we can find a strongly $\sigma$-invariant state $\Psi$ on $\B(\rL^2(M))$ which extends $\omega_{\theta}$. By Theorem \ref{thm state strong invariant}, we have $\Psi \in \conv\mathcal{E}(\log \Delta)$ and by Lemma \ref{lem absorb}, we can in fact assume that $\Psi \in \conv\mathcal{E}_0(\log \Delta)$. Then $\Psi$ is the barycenter of some probability measure $\mu$ on $\mathcal{E}_0(\log \Delta)$. Hence we can find $\psi$ in the support of $\mu$ such that
$$ \sum_k \psi(|\rho(b_k)|^2) \geq \frac{1}{2} \sum_k \Psi(|\rho(b_k)|^2)=\frac{1}{2} \sum_k \|\xi_k\|^2. $$
 Since $\psi \in \mathcal{E}_0(\log\Delta)$, we can find a net $(x_i)_{i \in I}$ in $M$ with $\|x_i\|_{\varphi_1}=1$ for all $i \in I$ such that $\lim_i \| x_i \varphi_1^{1/2}-\varphi_2^{1/2}x_i \|=0$ and $\psi=\lim_i \langle \bdot x_i \varphi_1, x_i\varphi_1 \rangle =0$ in the weak$^*$ topology. 
 Then we get
 $$ \lim_i \sum_k \|x_i \xi_k\|^2=\sum_k \psi(|\rho(b_k)|^2) \geq \frac{1}{2} \sum_k \|\xi_k\|^2.$$
  Since for all $k$, we have 
$$\Psi(|\lambda(\theta(a_k))-\rho(b_k)|^2)=\|a_k\varphi_1^{1/2}-\varphi^{1/2}_1b_k\|^2=0,$$
then we also have $\psi(|\lambda(\theta(a_k))-\rho(b_k)|^2)=0$. Thus for all $k$, we get
$$\lim_i\| \theta(\xi_k)x_i-x_i \xi_k\|^2=\lim_i \| \theta(a_k) \varphi_2^{1/2}x_i-x_i \varphi_1^{1/2}b_k\|^2=\psi(|\lambda(\theta(a_k))-\rho(b_k)|^2)=0.$$
which means that if $i$ is large enough, we will have
$$ \sum_k \|x_i \xi_k-\theta(\xi_k)x_i \|^2 \leq \varepsilon \sum_k \|x_i \xi_k\|^2.$$
By \cite[Theorem III.1]{Co85}, we conclude that $\theta$ is approximately inner.
\end{proof}

Let $M$ be a type $\III_\lambda$ factor for $\lambda \in ]0,1[$ with $\lambda$-trace $\varphi$. Then for any $\theta \in \Aut(M)$ such that $\mathrm{mod}(\theta)$ is trivial, we can find a unitary $u \in M$ such that $\theta \circ \Ad(u)$ leaves $\varphi$ invariant. Hence the same proof of Theorem \ref{main weakly inner} combined with Remark \ref{type lambda} gives the following.

\begin{thm} \label{thm weak inner lambda}
Let $M$ be a factor of type $\III_\lambda, \: \lambda \in ]0,1[$ such that $M \cong M \ovt R_\lambda$. Then a weakly inner automorphism $\theta \in \Aut(M)$ is approximately inner if and only if $\mathrm{mod}(\theta)$ is trivial.
\end{thm}

We end this section with the following consequence of Theorem \ref{main full compact}.

\begin{prop} \label{prop weak inner full}
Let $M$ be a full factor. Then every weakly inner automorphism of $M$ is inner.
\end{prop}
\begin{proof}
Let $\theta$ be a weakly inner automorphism of $M$. Then there exists an automorphism $\alpha$ of $C^*_{\lambda \cdot \rho}(M)$ such that $\alpha(\lambda(a)\rho(b))=\lambda(\theta(a))\rho(b)$ for all $a,b \in M$. Let $\varphi$ be a faithful normal state on $M$. Since $M$ is full, we know by Theorem \ref{main full compact} that $C^*_{\lambda \cdot \rho}(M)$ contains the rank one projection $e_\varphi$ on $\varphi^{1/2}$. For every $\varphi$-analytic $a \in M$ we have $\lambda(a)e_\varphi=\rho(\sigma^{\varphi}_{\ri/2}(a))e_\varphi$. Therefore, by applying $\alpha$, we obtain $\lambda(\theta(a))\alpha(e_\varphi)=\rho(\sigma^{\varphi}_{\ri/2}(a))\alpha(e_\varphi)$. Let $\eta$ be a unit vector in the image of the rank one projection $\alpha(e_\varphi)$. Then we have $\theta(a)\eta=\eta \sigma^{\varphi}_{\ri/2}(a)$ for all $\varphi$-analytic $a$. So by taking the adjoint, we get $\sigma_{-\ri/2}^\varphi(a^*)\eta^*=\eta^*\theta(a^*)$ for all $\varphi$-analytic $a$. Replace $a^*$ by $\sigma_{\ri/2}^\varphi(a)$. We get $a\eta^*=\eta^*\theta(\sigma_{\ri/2}(a))$ for all $\varphi$-analytic $a$. Therefore, if we let $\psi=\eta^*\eta$, we have
$$ a\psi=\eta^*\theta(\sigma_{i/2}(a))\eta=\eta^*\eta \sigma_{i}^\varphi(a)=\psi \sigma_{i}^\varphi(a)$$
 for all $\varphi$-analytic $a$. This forces $\psi=\varphi$ and $\eta=u\varphi^{1/2}$ for some unitary $u \in M$. We conclude easily that $\theta=\Ad(u)$.
\end{proof}

\bibliographystyle{plain}

\end{document}